\documentclass[12pt,a4paper,reqno]{amsart}
\usepackage{mathrsfs}
\usepackage{bbm}
\RequirePackage{color}
\RequirePackage[usenames,dvipsnames]{xcolor}

\usepackage[margin=2.4cm]{geometry}
\usepackage{graphics, enumitem, comment}
\usepackage{graphicx}
\usepackage{epsfig}

\usepackage{amssymb}
\usepackage{amsthm}
\usepackage{colortbl}

\newif\ifcolorcomments

\def\bc{\begin{center}}
\def\ec{\end{center}}
\def\be{\begin{equation}}
\def\ee{\end{equation}}
\def\F{\mathcal F}

\def\e{\mathbb E}
\def\Q{\mathcal Q}
\def\P{\mathbb{P}}

\def\H{\mathcal H}

\newcommand\hdim{\dim_{\mathrm H}}
\newcommand\pdim{\dim_{\mathrm P}}

\usepackage{amsmath}

\newtheorem{lem}{Lemma}[section]

\newtheorem{prop}{Proposition}[section]
\newtheorem{theorem}{Theorem}[section]
\newtheorem{lemma}[lem]{Lemma}

\newtheorem{corollary}{Corollary}[section]

\theoremstyle{remark}
\newtheorem{remark}{Remark}[section]

\numberwithin{equation}{section}

\newif\ifdraft\drafttrue

\makeatletter
\renewcommand\@biblabel[1]{#1}
\makeatother

\makeatletter
\@namedef{subjclassname@2020}{2020 Mathematics Subject Classification}
\makeatother

\begin{document}

\subjclass[2020] {}

\title{On the hitting probabilities of limsup random fractals }

\author[Zhang-nan Hu]{Zhang-nan Hu}
\address{Zhang-nan Hu, School of Mathematics, South China University of Technology, Guangzhou, 510641, China}
\email{hnlgdxhzn@163.com}
\author[Wen-Chiao Cheng]{Wen-Chiao Cheng}
\address{Wen-Chiao Cheng, Department of Applied Mathematics, Chinese Culture University, Yangmingshan, Taipei, Taiwan}
\email{zwq2@faculty.pccu.edu.tw}
\author{Bing Li*}\thanks{* Corresponding author}
\address{Bing Li, School of Mathematics, South China University of Technology, Guangzhou, 510641, China}
\email{scbingli@scut.edu.cn}

\begin{abstract}

Let $A$ be a limsup random fractal with indices $\gamma_1, ~\gamma_2 ~$and $\delta$ on $[0,1]^d$. We determine the hitting probability $\P(A\cap G)$ for any analytic set $G$ with the condition $(\star)$$\colon$ $\hdim(G)>\gamma_2+\delta$, where $\hdim$ denotes the Hausdorff dimension.
This extends the correspondence of Khoshnevisan, Peres and Xiao 
\cite{KPX} by relaxing the condition that the probability $P_n$ of choosing each dyadic hyper-cube is homogeneous and $\lim\limits_{n\to\infty}\frac{\log_2P_n}{n}$ exists.
We also present some counterexamples to show the Hausdorff dimension in condition $(\star)$ can not be replaced by the packing dimension.  
\end{abstract}

\maketitle

\section{Introduction}

A limsup random fractal is a type of limsup set induced by a random model defined on the unit cube $[0,1]^d$ of Euclidean space. Let $\mathcal{Q}_n$ denote the collection of $d$-dimensional dyadic hyper-cubes in $[0,1]^d$ for $n\ge0$, that is
\[\mathcal{Q}_n=\{ [k_12^{-n},(k_1+1)2^{-n}]\times\dots\times[k_d2^{-n},(k_d+1)2^{-n}]\colon 0\le k_i\le 2^{n}-1,~1\le i\le d\}.\]
Let $\Q=\bigcup_{n\ge0}\Q_n$. For $n\ge1$, let $\{Z_n(Q),Q\in \mathcal{Q}_n\}$ be a collection of random variables defined on a probability space $(\Omega,\mathscr{B},\P)$, each taking values in $\{0,1\}$. We say that $Q\in\Q_n$ is chosen if $Z_n(Q)=1$.
Let
\[A(n)=\bigcup_{Q\in \mathcal{Q}_n,\atop Z_n(Q)=1}{Q}.
\]
That is, the union of chosen dyadic cubes of order $n$. The random set
\[A=\limsup_{n\to\infty}A(n)\]
is called a \textbf{limsup~ random~ fractal} associated to $\{Z_n(Q),n\ge 1,Q\in \mathcal{Q}_n\}$.
For $n\ge1$ and $Q\in\Q_n$, denote $P_n(Q)=\P(Z_n(Q)=1)$, and write
\begin{equation}\label{pn1}
\gamma_1:=-\limsup_{n\to\infty}\frac{\log_{2}(\max_{Q\in\Q_n}P_n(Q))}{n},
\end{equation}
\begin{equation}\label{pn2}
\gamma_2:=-\limsup_{n\to\infty}\frac{\log_{2}(\min_{Q\in\Q_n}P_n(Q))}{n}.
\end{equation}
We adopt the convention that $\log_20=-\infty$.

We assume the following condition, which is a slight modification of Condition 5 in \cite{KPX}.

{\bf Correlation Condition}:\, there is a constant $\delta\ge0$ such that for all $\epsilon>0$,
\begin{equation}\label{ast}
    \limsup_{n\to\infty}\frac{1}{n}\log_2f(n,\epsilon)\le\delta,
\end{equation}
where
\[f(n,\epsilon)=\max_{Q\in \mathcal{Q}_n}\#\{Q'\in \mathcal{Q}_n\colon{\rm Cov}(Z_n(Q),Z_n(Q'))\ge\epsilon P_n(Q)P_n(Q')\},\]
and ${\rm Cov}(X, Y)$ is the covariance of random variables $X$ and $Y$. We refer to $\gamma_1$, $\gamma_2$ and $\delta$ as the indices of the limsup random fratcal $A$.

In 2000, Dembo, Peres, Rosen and Zeitouni \cite{DPRZ} established a lower bound for the Hausdorff dimension of limsup random fractals. Khoshnevisan, Peres and Xiao \cite{KPX} studied
the hitting probability of limsup random fractals with the conditions that the probability $P_n(Q)$ does not depend on $Q$, denoted by $P_n$, and $\lim\limits_{n\to\infty}\frac{\log_2P_n}{n}$ exists, and they obtained an estimate of Hausdorff dimension of the intersection of the limsup random fractal and an analytic set.
 In 2013, Zhang \cite{Zhang} 
 determined the Hausdorff dimension of limsup random fractals under the independence of $\{Z_n(Q),n\ge 1,Q\in \mathcal{Q}_n\}$.  Hu, Li and Xiao \cite{HLX} extended the results in \cite{KPX} on limsup random fractals to metric spaces.
 They proved that 
under a Correlation Condition, if the limits in (\ref{pn1}) and (\ref{pn2}) exist, that is 
\begin{equation}\label{lim1}
\gamma_1=-\lim_{n\to\infty}\frac{\log_{2}(\max_{Q\in\Q_n}P_n(Q))}{n}, \quad
\gamma_2=-\lim_{n\to\infty}\frac{\log_{2}(\min_{Q\in\Q_n}P_n(Q))}{n},
\end{equation}
 then for an analytic set $G$,
\begin{align}
 \P(A\cap G\ne\varnothing)&=0 ~ \text{if $\dim_{\rm P}(G)<\gamma_1$,}\label{22}\\
 \P(A\cap G\ne\varnothing)&=1~  \text{if $\dim_{\rm P}(G)>\gamma_2+\delta$,} \label{23}
\end{align}
where $\pdim$ denotes the packing dimension.

There are many random sets 
which are also closely related to limsup random fractals, such as the fast points of Brownian motion \cite{OT},
thick points of Brownian motion \cite{DPRZ}, random covering sets \cite{Dvor56}, dynamical covering sets \cite{FanScT13}, shrinking target sets \cite{Yuan} and so on. 
Li, Shieh and Xiao \cite{LSX} investigated the hitting probability of random covering sets in which the use of limsup random fractals is essential.
Later Li and Suomala \cite{LS} studied the same problem under conditions different from those in \cite{LSX}.
Wang, Wu and Xu \cite{Wang} considered the dynamical covering problems on the middle-third Cantor set. Hu, Li and Xiao \cite{HLX}
applied the methods of limsup random fractals to investigate the intersection property of dynamical covering sets. Lyons \cite{lyo} studied the percolation on trees, and gave the Hausdorff dimension of limsup random fractals. 
Fan \cite{fan} studied the sets of limsup deviation paths on trees similar to limsup random fractals, and gave their Hausdorff dimensions. For further results concerning stochastic process on trees, see  \cite{fanf}, \cite{climb}.

In this paper, we are interested in whether both (\ref{22}) and (1.6) hold if the limits in (\ref{lim1}) do not exist.
Our answer for this question is that (\ref{22}) holds, but (\ref{23}) does not. 
 Hence if we replace $\dim_{\rm P}(G)>\gamma_2+\delta$ by the stronger condition $\hdim(G)>\gamma_2+\delta$, then (\ref{23}) will hold, as is shown by the following theorem, which is the main theorem of this paper.

Throughout the paper,  $\dim_{\rm H}$ and $\overline{\dim}_{\rm B}$ denote the Hausdorff dimension and upper box dimension, respectively.

\begin{theorem}\label{limsup1}
Let $A=\limsup\limits_{n\to\infty}A(n)$ be a limsup random fractal with indices $\gamma_1$, $\gamma_2$, and satisfying the Correlation Condition with $\delta\ge0$. Then for any analytic set $G\subset [0,1]^d$,
\[\P(A\cap G\ne\varnothing)=
    \begin{cases}
    0&  \text{if $\dim_{\rm P}(G)<\gamma_1$,}\\
    1&  \text{if $\dim_{\rm H}(G)>\gamma_2+\delta$.}
    \end{cases}\]
\end{theorem}
\begin{remark}
If $\gamma_2+\delta<d$, by Theorem \ref{limsup1}, we see that with probability one, for any hyper cube $Q\in\Q$, $A\cap Q\ne\varnothing$ almost surely (a.s. for short). It implies that $A$ is dense in $[0,1]^d$ a.s. Hence under the condition $\gamma_2+\delta<d$, $\dim_{\rm P} (A)=d$ a.s.
\end{remark}
Hu, Li and Xiao \cite{HLX} provided an estimate for the Hausdorff dimension of $A\cap G$ requiring the existence of the limits in 
(\ref{lim1}), 
and the following corollary of Theorem \ref{limsup1} indicates the existence of the limits can be relaxed.
\begin{corollary}\label{cor}
Under the same conditions as Theorem \ref{limsup1}, suppose $\delta=0$. Then we have
\[\P(A\cap G\ne\varnothing)=
    \begin{cases}
    0&  \text{if $\dim_{\rm P}(G)<\gamma_1$,}\\
    1&  \text{if $\dim_{\rm H}(G)>\gamma_2$,}
    \end{cases}\]
and
\begin{equation}\label{uplow}
\max\{0, \hdim (G)-\gamma_2\}\le \hdim(A\cap G)\le \max\{0, \pdim (G)-\gamma_1\}\quad {\rm a.s.}
 \end{equation}
In particular, if $\gamma_1=\gamma_2$, then $\hdim (A)=\max\{0, d-\gamma_1\}$ {\rm a.s. }
\end{corollary}

In Theorem \ref{limsup1}, the condition for probability one is that $\hdim(G)>\gamma_2+\delta$, which can be  weakened, by (\ref{23}), to $\pdim(G)>\gamma_2+\delta$ 
 if (\ref{lim1}) holds. A natural question is whether $\hdim(G)$ in Theorem \ref{limsup1} can be replaced by $\pdim(G)$ or not. The answer is negative.

\begin{prop}\label{exam}
There exist a closed set $G\subset [0,1]$ with $\pdim (G)=1$ and a limsup random fractal $A$ with $\gamma_1=\gamma_2=\delta=0$ such that $\P(A\cap G\ne\varnothing)=0$.
\end{prop}

Corollary \ref{cor} provides estimates for $\hdim(A\cap G)$. But both inequalities in (\ref{uplow}) can be strict.

\begin{prop}\label{exam2}
For $\gamma_0,~t\in[0,1]$, there exist a closed set $G\subset [0,1]$ with $\hdim (G)=t$ and $\pdim (G)=1$, and a limsup random fractal $A$ with $\gamma_1=\gamma_2=\gamma_0$, $\delta=0$ such that $\hdim(A\cap G)=\min\{t, 1-\gamma_0\}$ a.s.
\end{prop}

\begin{remark}\label{remark}
We notice that the condition $\pdim (G)<\gamma_1$ for zero probability  in Theorem \ref{limsup1} can not be replaced by $\hdim (G)<\gamma_1$. 
Indeed given $\gamma_0 \in (0, 1)$ and $t \in(0, \gamma_0)$, by Proposition \ref{exam2}, there is a limsup random fractal $A$ with indices $\gamma_1 = \gamma_0$ and a set $G$ with $\dim_{\rm H}(G) =t< \gamma_1$ such that $\dim_{\rm H}(A\cap G)>0$ a.s., implying $A\cap G\ne\varnothing$ a.s.
\end{remark}

\section{Proofs of main results}

Before proving the results,
we first fix some notation. Write $f_n\lesssim g_n$, $n\in\mathcal{I}$, if there is an absolute constant $0<c<\infty$ such that for all $n\in\mathcal{I}$, $f_n\le cg_n$.
If $f_n\lesssim g_n$ and $g_n\lesssim f_n$ for $n\in\mathcal{I}$, then we denote $f_n\asymp g_n$. $\chi$ is the indicator function. Let $\mathcal{Q}'_n=\{ [k_12^{-n},(k_1+1)2^{-n})\times\dots\times[k_d2^{-n},(k_d+1)2^{-n})\colon 0\le k_i\le 2^{n}-1,~1\le i\le d\}$
 and $\Q'=\bigcup_{n\ge1}\Q'_n$. 

We demonstrate Theorem \ref{limsup1} by using the following lemmas.
For any $r>0$, let $\mathcal{C}_r(G)$ be a collection of the smallest number of closed balls with radius $r$ covering $G$, and $N_{r}(G)=\#\mathcal{C}_r(G)$. From \cite{Tri}, the upper box dimension of $G$ is defined as $\overline{\dim}_{\rm B}(G)=\limsup_{r\to0}\frac{\log N_r(G)}{-\log r}$.

The following lemma is a slight modification of Lemma 3.2 in \cite{HLX}. 
\begin{lem}[\cite{HLX}]\label{haus1}
Let $G\subset [0,1]^d$ be an analytic set. If $\dim_{\rm H}(G)>t$, there is a nonempty Borel subset $G_{\star}\subset G$ such that
\[\dim_{\rm H}(G_{\star}\cap V)>t\]
 for all dyadic hyper-cubes $V\in\Q'$ with $G_{\star}\cap V\ne\varnothing$.
\end{lem}
\begin{proof}
Since $G$ is analytic, from \cite{How}, there is a closed set $K \subset G$ with $0<\H^t(K) <\infty$
for some $t > s$. Let 
\[U = \bigcup_{V\in\Q\atop \dim_{\rm H}(V \cap K) \le s}V.\]
Let $G_{\star} = K \setminus U$ which is a Borel set and $\dim_{\rm H}(V \cap G_{\star}) > s$ for all $V\in\Q'$ intersecting $G_{\star}$. Moreover
\begin{equation*}
\begin{split}
\H^t(K)&=\H^t(K\cap U\cup K\setminus U))\\
&\le \sum_{V\in\Q'\atop \dim_{\rm H}(V \cap K) \le s}\H^t(K\cap V)+\H^t(G_{\star})\\
&=\H^t(G_{\star})
\end{split}
\end{equation*}
Hence $\H^t(G_{\star})=\H^t(K)>0$, implying $G_{\star}\ne\varnothing$.
\end{proof}

\begin{lem}[\cite{LS}]\label{haus2}
For $G$ with $\dim_{\rm H}(G)>t$, there is $N\ge1$ such that for $n\ge N$, there are at least $2^{nt}$ elements in
$\mathcal{Q}'_n$ intersecting G, that is
\[\#\big\{Q\in\mathcal{Q}'_n\colon Q\cap G\ne\varnothing\big\}\ge 2^{nt}.\]
\end{lem}

\begin{proof}[Proof of Theorem \ref{limsup1}]
  (i) The proof of the zero probability is known (see \cite{HLX}) and is presented for the sake of completeness. It suffices to show that $\overline{\dim}_{\rm B}(G)<\gamma_1$ implies $A\cap G=\varnothing$ a.s.
Indeed for $\dim_{\rm P}(G)<\gamma_1$, by
\begin{equation}\label{dimp}
\pdim(G)=\inf\Big\{\sup_n\overline{\dim}_{\rm B}(G_n)\colon G\subset \bigcup_{n=1}^{\infty}G_n\Big\},
\end{equation}
 (see \cite{Tri}), there is a covering $\{G_n\}$ with ${\overline\dim_{\rm B}}(G_n)<\gamma_1$ for $n\ge1$.  Hence $\P(A\cap G\ne\varnothing)\le \sum_n\P(A\cap G_n\ne\varnothing)=0.$

When $\gamma_1<\infty$,  fix $\epsilon>0$ with $\overline{\dim}_{\rm B}(G)<\gamma_1-\epsilon$, and $\theta\in(\overline{\dim}_{\rm B}(G),\gamma_1-\epsilon)$, then for $n$ large enough, we have 
\[\max_{Q\in\Q_n}P_n(Q)\le 2^{-n(\gamma_1-\epsilon)},\quad N_{\sqrt{d}2^{-n-1}}(G)\lesssim 2^{n\theta },\]
\[\#\Gamma_n(B)=\#\big\{Q\in\mathcal{Q}_n\colon Q\cap B\ne\varnothing\big\}\le M\quad (\forall B\in \mathcal{C}_{\sqrt{d}2^{-n-1}}(G)),\]
where $M>0$ is an absolute constant, and these inequalities are immediate from their definitions. 
Hence for $n$ large enough, we have

\[\P(G\cap A(n)\ne\varnothing)\le \P\Big(\bigcup_{B\in \mathcal{C}_{\sqrt{d}2^{-n-1}}(G)}\bigcup_{Q\in \Gamma_n(B)}Q\cap A(n)\ne\varnothing\Big)\lesssim 2^{-n(\gamma_1-\epsilon-\theta)}.\]
The Borel-Cantelli lemma implies $A\cap G=\varnothing$ a.s. 

 When $\gamma_1=\infty$, fix $b>0$ with $\overline{\dim}_{\rm B}(G)<b$. For $\theta\in(\overline{\dim}_{\rm B}(G),b)$, we have
$N_{\sqrt{d}2^{-n-1}}(G)\lesssim 2^{n\theta }$ and $\max_{Q\in\Q_n}P_n(Q)\le 2^{-nb}$ for $n$ large enough, where the second inequality follows from (\ref{pn1}). Similarly, we get $A\cap G=\varnothing$ a.s.

(ii)   In the following, we will show  $\P(A\cap G\ne\varnothing)=1$ provided that $\dim_{\rm H}(G)>\gamma_2+\delta$.
 By Lemma \ref{haus1}, there exist a Borel subset $G_{\star}$ and a closed subset $K$ satisfying $G_{\star}\subset K\subset G$ such that for all $V\in\Q'$, whenever $G_{\star}\cap V \ne \varnothing$, ${\dim}_{\rm H}(G_{\star}\cap V)>\gamma_2+\delta$.
    Fix an open set $V$ such that $V\cap G_{\star}\ne \varnothing$. Denote
    \[\mathcal{A}_n=\mathcal{A}_n(V\cap G_{\star})=\big\{Q\in \mathcal{Q}'_n\colon Q\cap V\cap G_{\star}\ne\varnothing\big\},\]
    and $N_n=\#\mathcal{A}_n$. Then by Lemma \ref{haus2}, we have that for any $\eta\in(\gamma_2+\delta,\dim_{\rm H}( G_{\star}\cap V))$, there is $n(\eta)\ge1$ such that
    \[N_n\ge 2^{n\eta}\quad (\forall n\ge n(\eta)).\]

    Define
    \[S_n=\sum_{Q\in\mathcal{A}_n}Z_n(\overline Q),\]
   where $\overline Q$ is the closure of $Q$ and $\overline Q\in\Q$. We need only show that $\mathbb{P}(S_n>0~\, \hbox{ i.o.})=1$, where i.o. stands for infinitely often. Firstly, we estimate
    \[{\rm Var}(S_n)=\sum_{ Q\in\mathcal{A}_n}\sum_{Q'\in \mathcal{A}_n}{\rm Cov}(Z_n(\overline Q),Z_n(\overline{Q'})).\]
    Fix $\epsilon>0$ and for each $Q\in \mathcal{A}_n$, let $\mathcal{G}_n(Q)$ denote the collection of all $Q'\in \mathcal{A}_n$ such that
    \begin{equation}\label{covv}
    {\rm Cov}(Z_n(\overline Q),Z_n(\overline{Q'}))\le \epsilon P_n(\overline Q)P_n(\overline{Q'}),
    \end{equation}
    and we define $\mathcal{B}_n(Q)=\mathcal{A}_n\setminus \mathcal{G}_n(Q)$.
    
    From the fact that ${\rm Cov}(Z_n(\overline Q),Z_n(\overline{Q'})\le \mathbb{E}(Z_n(\overline Q)=P_n(\overline Q)$, we get
   \begin{equation*}
        \begin{split}
        {\rm Var}(S_n)&=\sum_{Q\in\mathcal{A}_n}\Big(\sum_{Q'\in\mathcal{G}_n(Q)}{\rm Cov}\big(Z_n(\overline Q),Z_n(\overline{Q'})\big)+\sum_{Q'\in\mathcal{B}_n(Q)}{\rm Cov}\big(Z_n(\overline Q),Z_n(\overline{Q'})\big)\Big)\\
        &\le \sum_{Q\in\mathcal{A}_n}\Big(\sum_{Q'\in\mathcal{G}_n(Q)}\epsilon P_n(\overline Q)P_n(\overline{Q'})+\sum_{Q'\in\mathcal{B}_n(Q)}P_n(\overline Q)\Big)\\
        &= \sum_{Q\in\mathcal{A}_n}\sum_{Q'\in\mathcal{G}_n(Q)}\epsilon P_n(\overline Q)P_n(\overline{Q'})+\sum_{Q\in\mathcal{A}_n}\big(\#\mathcal{B}_n(Q)\big)P_n(\overline Q) \\
        &\le \epsilon \bigg(\sum_{Q\in\mathcal{A}_n}P_n(\overline Q)\bigg) \bigg(\sum_{Q'\in\mathcal{A}_n}P_n(\overline{Q'})\bigg)
        + \bigg(\max_{Q\in \mathcal{A}_n}\#\mathcal{B}_n(Q)\bigg)\bigg(\sum_{Q\in\mathcal{A}_n}P_n(\overline Q)\bigg).\\
        \end{split}
    \end{equation*}
    Recalling the notation
    of the Correlation Condition, we have
    \begin{equation}\label{Var}
    \begin{split}
    {\rm Var}(S_n)&\le  \epsilon \bigg(\sum_{Q\in\mathcal{A}_n}P_n(\overline Q)\bigg)^2+f({n},\epsilon)\sum_{Q\in\mathcal{A}_n}P_n(\overline Q)\\
   & =\epsilon \big(\mathbb{E}(S_n)\big)^2 +f({n},\epsilon)\mathbb{E}(S_n).
    \end{split}
    \end{equation}
  Combining (\ref{Var}) and the Paley-Zygmund inequality \cite{Kahane}, we obtain
   \begin{equation*}\label{sn}
        \begin{split}
        \P(S_n>0)&\ge \frac{\bigl(\mathbb{E}(S_n)\bigr)^2}{\mathbb{E}(S_n^2)}=\frac{\bigl(\mathbb{E}(S_n)\bigr)^2}{\bigl(\mathbb{E}(S_n)\bigr)^2+{\rm Var}(S_n)}\\
       & \ge \frac{\mathbb{E}(S_n)}{(1+\epsilon)\mathbb{E}(S_n)+f(n,\epsilon)}.
        \end{split}
    \end{equation*}
   That is,
    \begin{equation}\label{sn1}
    \P(S_n>0)\ge\frac{1}{1+\epsilon+\frac{f(n,\epsilon)}{\mathbb{E}(S_n)}}.
    \end{equation}
    Since $\mathbb{E}(S_n)=\sum_{Q\in\mathcal{A}_n}P_n(\overline Q)\ge N_n(\min_{\overline Q\in\Q_n} P_n(\overline Q))$, recalling that $N_n=\#\mathcal{A}_n$, we have
    \begin{equation}\label{sn2}
    \frac{f(n,\epsilon)}{\mathbb{E}(S_n)}\le \frac{f(n,\epsilon)}{N_n(\min_{\overline Q\in\Q_n} P_n(\overline Q))}.
    \end{equation}
    By the Correlation Condition, for any $\theta>0$
    with $2\theta<\eta-\delta-\gamma_2$, we have
    \[f({n},\epsilon)\le 2^{(\delta+\theta)n}\]
    for $n$ large enough, and from (\ref{pn2}), we have
    \[\min_{\overline Q\in\Q_n}P_n(\overline Q)\ge 2^{-(\gamma_2+\theta)n}\]
    for infinitely many $n$, denoted by $\mathcal{N}$.
    From these inequalities and the arbitrariness of $\theta$, we have
    \begin{equation}\label{si0}
    \limsup_{ n\in \mathcal{N}\atop n\to\infty}\frac{f(n,\epsilon)}{N_n(\min_{\overline Q\in\Q_n} P_n(\overline Q))}
    \le \limsup_{ n\in \mathcal{N}\atop n\to\infty}2^{n(2\theta+\delta+\gamma_2-\eta)}=0.
    \end{equation}
    Then combining inequalities (\ref{sn1})--(\ref{si0}), by Fatou's lemma  and arbitrariness of $\epsilon$, we conclude that 
    \[\P(S_n>0~\, \hbox{ i.o.} )\ge \limsup_{ n\in \mathcal{N}\atop n\to\infty}\P(S_n>0)=1.\]
It follows that
	\[ \mathbb{P}\Big( \bigcup_{n=k}^{\infty}\bigcup_{ Q\in\mathcal{A}_n} \{Z_n(\overline Q)=1\},~\forall k\ge1, ~\forall V\in\mathcal{Q}'\Big)=1.\]
	Then given $\omega$ in the above event, for $Q_0=[0,1)^d$, there is some $k_1\ge1$ such that  
	\[\exists~ Q_1\in\Q'_{k_1},\quad Q_1\cap G_{\star}\cap Q_0\ne\varnothing\quad {\rm and}\quad Z_{k_1}(\overline{Q_1})(\omega)=1.\]
	Since $Q_1\subset Q_0$ and $G_{\star}\cap Q_1\ne\varnothing$, there is some $k_2\ge k_1$ such that 
	\[\exists~ Q_2\in\Q'_{k_2}\quad Q_2\cap G_{\star}\cap Q_1\ne\varnothing\quad{\rm and}\quad Z_{k_2}(\overline{Q_2})(\omega)=1. \]
       We continue inductively, then get a sequence $\{Q_i\}$ such that for all $i\ge1$, 
	\[Q_i\in\Q'_{k_i}\quad  Q_i\cap G_{\star}\cap Q_{i-1}\ne\varnothing\quad {\rm and}\quad Z_{k_i}(\overline{Q_i})(\omega)=1.\]
	Notice that $\{Q_i\}$ are decreasing hyper cubes. Recall that $K\supset G_{\star}$ is closed. Hence $\varnothing\ne\bigcap_{i=1}^{\infty}\overline{Q_i}\subset K\cap A(\omega)$. Thus $A\cap G\ne\varnothing$ a.s. 

\end{proof}

We use the following lemma to prove Corollary \ref{cor}.
\begin{lem}[\cite{KPX}]\label{lemp}
Suppose $A=A(\omega)$ is a random set in $[0,1]^d$ (i.e. $\chi_{A(\omega)}(x)$ is jointly measurable) such that for any compact set $E\subset [0,1]^d$ with $\hdim(E)>\gamma$, we have $\P(A\cap E)=1$. Then for any analytic set $E\subset [0,1]^d$,
\[\hdim(E)-\gamma\le\hdim(A\cap E)\quad a.s.\]
\end{lem}

\begin{proof}[Proof of Corollary \ref{cor}]
The proof of the upper bound is similar to that in \cite{KPX}, and is included for completeness. Since $\pdim(G)<\gamma_1$ implies $A\cap G=\varnothing $ a.s.,
we consider the case of $\pdim(G)\ge\gamma_1$. By (\ref{dimp}), it suffices to prove $\hdim(A\cap G)\le \overline{\dim}_{\rm B}(G)-\gamma_1$. Recall the notation in Theorem \ref{limsup1} and the start of Section 2, and define
 \[H_n=\sum_{B\in\mathcal{C}_{\sqrt{d}2^{-n-1}}(G)}\sum_{Q\in\Gamma_n(B)}Z_n(Q), \]
For $\eta=\overline{\dim}_{\rm B}(G)-\gamma_1$ and any $\epsilon>0$,  for large enough $n$ we have
\[\mathbb{E}(H_n)=\sum_{B\in\mathcal{C}_{\sqrt{d}2^{-n-1}}(G)}\sum_{Q\in\Gamma_n(B)}P_n(Q)\lesssim 2^{n(\eta+2\epsilon)}. \]
Due to the arbitrariness of $\epsilon$, for any $\theta>\eta$, we have $\mathbb{E}(\sum_{n=1}^{\infty}H_n2^{-n\theta})<\infty$. Therefore,
\[\mathcal{H}^{\theta}(A\cap G)\le \liminf_{m\to\infty}\sum_{n=m}^{\infty}H_n2^{-n\theta}=0\quad {\rm a.s.},\]
giving $\hdim(A\cap G)\le \overline{\dim}_{\rm B}(G)-\gamma_1$ a.s.

The left-hand inequality in Corollary \ref{cor} follows from Lemma \ref{lemp} and Theorem \ref{limsup1}.

\end{proof}

\begin{proof}[Proof of Proposition \ref{exam}]
Let $(t_k)_{k\ge1}$ be an increasing sequence with $t_k\in(0,1)$ and $\lim\limits_{k\to\infty}t_k=1$. Let $(n_k)_{k\ge1}$ and $(m_k)_{k\ge1}$ be increasing sequences of positive integers with $m_k<n_k$ (to be determined later).
We first construct a homogeneous Cantor set $G$ as follows. We divide $[0,1]$ into $N_1:=2^{m_1}$ closed intervals, and inside each interval, choose a  closed subinterval with length $l_1:=2^{-n_1}$. Denote this collection of intervals by $\mathcal{G}_1$. Let
$I_1^1, ~I_1^2,\dots, ~I_1^{N_1}$ be the intervals in $\mathcal{G}_1$, arranged from left to right, 
 such that $I_1^1$ has the same  left endpoint as $[0,1]$, and $I_1^{N_1}$ has the same right endpoint as $[0,1]$.

Suppose $\mathcal{G}_k$ has been chosen, which is comprised of some closed intervals, each with length $l_k$. For any $I\in\mathcal{G}_k$, we divide $I$ into $2^{m_{k+1}}$ intervals of length $2^{-m_{k+1}}|I|$. Then, inside of each subinterval of length $2^{-m_k}l_k$, we choose one closed interval of length $l_{k+1}:=2^{-n_{k+1}}l_k$
such that the first subinterval has the same left endpoint as $I$, and the last subinterval has the same right endpoint as $I$.
Denote these $N_{k+1}=2^{m_{k+1}}N_k$ intervals with length $l_{k+1}$ by $\mathcal{G}_{k+1}$. Let $G_k=\bigcup_{I\in \mathcal{G}_k}I$, and $G=\bigcap_{k\ge1} G_k$.
Choosing $m_{k+1}$ large enough (depending on the choices of $(n_i)_{1\le i\le k}$ and $(m_i)_{1\le i\le k}$) such that $2^{m_{k+1}(1-t_k)}N_kl_k^{t_k}\ge1$, from \cite{FWW}, we have
\begin{equation*}
\begin{split}
\pdim (G)&=\limsup_{k\to\infty}\frac{\log N_{k+1}}{-\log (l_k)+\log (N_{k+1}/N_k)}\\
&=\limsup_{k\to\infty}\frac{\log (2^{m_{k+1}}N_k)}{\log (l_k^{-1}2^{m_{k+1}})}\ge \limsup_{k\to\infty}t_k=1.
\end{split}
\end{equation*}

Before constructing a limsup random fractal, we fix some notation. Let $(\xi_m)_{m\ge1}$ be a sequence of independent random variables on $(\Omega,\mathscr{B},\P)$ which are uniformly distributed on $[0,1]$.
Let $M_k=\sum_{i=1}^kn_i$, and $b_n=2^{M_kt_k}$ if $n\in[M_k,M_{k+1})$. Define the set
$B=\{\exists m \in [b_{n-1},b_n){\rm ~such ~that ~} \xi_m\in Q\}$ and set $Z_n(Q)=\chi_B$. Let $A(n)=\bigcup_{Q\in\mathcal{Q}_n\atop Z_n(Q)=1}Q$,  and $A=\limsup\limits_{n\to\infty}A(n).$

For every $Q\in\Q_n$, it follows from our assumption of $(\xi_m)_{m\ge1}$ and the definition of $\Q_n$ that $\P(Z_n(Q)=1)$ does not depend on $Q$, denoted by $P_n$.
Therefore,
 \begin{equation}\label{up}
 P_n=\P\Big(\bigcup_{m=b_{n-1}}^{b_n-1}\xi_m\in Q\Big)\le \sum_{m=b_{n-1}}^{b_n-1}\P(\xi_m\in Q)=2^{-n}(b_n-b_{n-1}),
 \end{equation}
 and
 \begin{equation}\label{low}
 \begin{split}
 P_n&\ge \sum_{m=b_{n-1}}^{b_n-1}\P(\xi_m\in Q)-\sum_{m=b_{n-1}}^{b_n-1}\sum_{m'=b_{n-1}\atop m'\ne m}^{b_n-1}\P(\xi_m\in Q, \xi_{m'}\in Q)\\
 &=\sum_{m=b_{n-1}}^{b_n-1}\P(\xi_m\in Q)\Big(1-\sum_{m'=b_{n-1}\atop m'\ne m}^{b_n-1}\P(\xi_m\in Q)\Big)\\
 &=2^{-n}(b_n-b_{n-1})(1-2^{-n}(b_n-b_{n-1})).
 \end{split}
 \end{equation}
Write $x_n=2^{-n}(b_n-b_{n-1})$. Note that for $k\ge1 $, if $M_k<n<M_{k+1}$, $x_n=0$, and $x_{M_k}=\frac{2^{M_kt_k}-2^{M_{k-1}t_{k-1}}}{2^{M_k}}$.
From this and inequalities (\ref{up}) and (\ref{low}), we have
 \begin{equation}\label{pn0}
  \limsup_{n\to\infty} \frac{\log_2P_n}{n}=0.
  \end{equation}

Next we prove that $A$ satisfies the  Correlation Condition with $\delta=0$. First we estimate ${\rm Cov}(Z_n(Q),Z_n(Q'))$ for $Q,Q'\in\Q_n$ with the distance ${\rm dist}(Q,Q')\ge 2^{-n}$,
\begin{equation}\label{cov}
\begin{split}
{\rm Cov}(Z_n(Q),Z_n(Q'))&=\e(Z_n(Q)Z_n(Q'))-\e(Z_n(Q))\e(Z_n(Q'))\\
&\le \P(Z_n(Q)=1,Z_n(Q')=1)-P_n^2\\
&\le \sum_{m=b_{n-1}}^{b_n-1}\P(\xi_m\in Q)\sum_{m'=b_{n-1}\atop m'\ne m}^{b_n-1}\P(\xi_m\in Q)-P_n^2\\
&\le 2(\sum_{m=b_{n-1}}^{b_n-1}\P(\xi_m\in Q))\sum_{m=b_{n-1}}^{b_n-1}\sum_{m'=b_{n-1}\atop m'\ne m}^{b_n-1}\P(\xi_m\in Q, \xi_{m'}\in Q')\\
& \lesssim 2^{-n+1}(b_n-b_{n-1})\e(Z_n(Q))\e(Z_n(Q')).
\end{split}
\end{equation}
From (\ref{pn0}), for $\epsilon>0$, there is $n(\epsilon)$ such that for any $n\ge n(\epsilon)$, $2^{-n+1}(b_n-b_{n-1})<\epsilon$. It implies that for $n$ large enough, we have $f(n,\epsilon)\le3$, hence $A$ satisfies the Correlation Condition with $\delta=0$.

Denote $A_k=\bigcup_{n=M_{k-1}+1}^{M_k}A(n)$, then
\begin{equation}\label{est1}
\begin{split}
\P(A_k\cap G_k\ne\varnothing)&\le \sum_{I\in \mathcal{G}_k}\P(I\cap A_k\ne\varnothing)\\
&\le N_kn_k\max_{M_{k-1}+1\le n\le M_k}\P(I\cap A(n)\ne\varnothing)\\
&\le 2N_kn_k\max_{M_{k-1}+1\le n\le M_k}P_n\lesssim N_kn_k2^{-M_k}2^{M_kt_k}\\
&=n_kN_kl_k2^{M_kt_k},
\end{split}
\end{equation}
and by choosing $n_k$ large enough (depending on $t_k$, the choice of $(n_i)_{1\le i<k}$ and $(m_i)_{1\le i\le k}$), we have $n_kN_kl_k2^{M_kt_k}\le 2^{-k}$. It then follows from the Borel-Cantelli lemma that
\[\P(A_k\cap G_k\ne\varnothing ~~i.o.)=0.\]
This means that there is some $N\ge1$ such that for any $k\ge N$, $A_k\cap G_k=\varnothing$ a.s. Since $G\subset G_k$, $k\ge 1$, it yields that $G\cap A(k)\ne\varnothing$ for only finitely many $k$. Therefore $\P(A\cap G\ne\varnothing)=0$.

\end{proof}

 To show Proposition \ref{exam2}, we make use of the following lemma. Here we suppose that $(\xi_m)_{m\ge1}$ is a sequence of independent and uniformly distributed random variables on $[0,1]$, which are defined on $(\Omega,\mathscr{B},\P)$.
  \begin{lemma}\label{lem1}
 Let $0\le \gamma_0<1$ and $a_n=2^{n(1-\gamma_0)}$. For any $Q\in\mathcal{Q}_n$, write $B=\{\exists~ m\in[a_{n-1},a_n)~s.t.~\xi_m\in Q\}$ and set $Z_n(Q)=\chi_{B}$. 
 For $0<\beta<1-\gamma_0$, let $$A^{\beta}(n)=\bigcup_{Q\in\mathcal{Q}_n\atop Z_n(Q)=1}{Q}^{\beta},$$ 
 where $Q^{\beta}$ is the interval with length $|Q|^{\beta}$ concentric with $Q$. Then
 \[\lim_{n\to\infty} {\rm \P}\Big(A^{\beta}(n)=[0,1]\Big)=1.\]
 \end{lemma}

\begin{proof}
For $n\ge1$, $Q\in\Q_n$, write
\[\Q_n(Q)=\big\{Q'\in\Q_n\colon {\rm dist}(Q,Q')\le 2^{-(n+1)}(2^{(1-\beta)n}-3) \big\},\]
hence $\#\Q_n(Q)=\lfloor2^{(1-\beta)n}\rfloor$. Note that \[\{ Q\nsubseteq A^{\beta}(n) \} \subset\{Z_n(Q')=0, {\rm for ~all ~}Q'\in\Q_n(Q)\}.\]
 Then for $n$ large enough, we have
\begin{equation}\label{1}
\begin{split}
\P\big(Q\nsubseteq A^{\beta}(n)\big )&\le \P\big(Z_n(Q')=0 {\rm ~for~all~} Q'\in\Q_n(Q)\big) \\
&\le \prod_{m=a_{n-1}}^{a_n}\P\Big(\xi_m\notin \bigcup_{Q'\in\Q_n(Q)}Q'\Big)\\
&=(1-\lfloor2^{(1-\beta)n}\rfloor 2^{-n})^{2^{n(1-\gamma_0)}}\\
&\le (1-2^{-n\beta-1})^{2^{n(1-\gamma_0)}}
 \end{split}
\end{equation}
Note that $\{[0,1]\nsubseteq A^{\beta}(n)\}= \{\exists ~Q\in\Q_n~{\rm such~that} ~ Q\nsubseteq A^{\beta}(n) \}$, and hence
\begin{equation}\label{2}
\begin{split}
\P\big([0,1]\nsubseteq A^{\beta}(n)\big)&=\P\big(\exists ~Q\in\Q_n~{\rm such~that} ~ Q\nsubseteq A^{\beta}(n) \big)\\
&\le\sum_{Q\in\Q_n}\P\big(Q\nsubseteq A^{\beta}(n) \big)\\
&\le \sum_{Q\in\Q_n}\P(Z_n(Q')=0{\rm ~for~all~} Q'\in\Q_n(Q)) \\
&\le \sum_{Q\in\Q_n}\prod_{m=a_{n-1}}^{a_n}\P\Big(\xi_m\notin \bigcup_{Q'\in\Q_n(Q)}Q'\Big)\\
&\le 2^n(1-2^{-n\beta-1})^{2^{n(1-\gamma_0)}},
 \end{split}
\end{equation}
which tends to 0 as $n\to\infty$.
 \end{proof}

 \begin{proof}[Proof of Proposition \ref{exam2}]

 Let $(n_k)_{k\ge1}$ be an increasing sequence of positive integers 
 (to be determined later). We divide $[0,1]$ into $N_1:=\lfloor2^{n_1t}\rfloor$ closed intervals, and inside each interval, choose a closed subinterval with length $l_1:=2^{-n_1}$,
 with the first subinterval on the left  having the same left endpoint as $[0,1]$,
 and the right endpoint of the last subinterval on the right having the same as that of $[0,1]$. We denote this collection of intervals by $\mathcal{G}_1$.

Suppose that $\mathcal{G}_k$, consisting of some closed intervals with length $l_k$, has been chosen. Then for any $I\in\mathcal{G}_k$, we divide $I$ into $\lfloor2^{n_{k+1}t}\rfloor$ closed intervals of length $2^{-\lfloor n_{k+1}t\rfloor}l_k$ and inside each of these, choose one closed interval of length $l_{k+1}:=2^{-n_{k+1}}l_k$, such that the left endpoint of the first subinterval is the same as that of $I$,
and the right endpoint of the last subinterval is the same as that of $I$.
Denote these $N_{k+1}=\lfloor2^{n_{k+1}t}\rfloor N_k$ intervals with length $l_{k+1}$ by $\mathcal{G}_{k+1}$. Let $G_k=\bigcup_{I\in \mathcal{G}_k}I$, and $G=\bigcap_{k\ge1} G_k$.
$G$ is then a homogeneous Cantor set. From \cite{FWW}, we have
\[\hdim (G)=\liminf_{k\to\infty}\frac{\log N_{k+1}}{-\log l_k}=\liminf_{k\to\infty}\frac{\log \prod_{i=1}^k\lfloor2^{n_it}\rfloor}{\log \prod_{i=1}^k 2^{n_i}}\le t,\]
and for $k$ large enough we also have
\[\frac{\log \prod_{i=1}^k\lfloor2^{n_it}\rfloor}{\log \prod_{i=1}^k 2^{n_i}}\ge\frac{\log \prod_{i=1}^k (2^{n_it}-1)}{\log \prod_{i=1}^k 2^{n_i}}=t+\frac{\log \prod_{i=1}^k (1-2^{-n_it})}{\sum_{i=1}^kn_i}
.\]
Since $\sum_{i=1}^{\infty} 2^{-n_it}<\infty$, it follows that $\prod_{i=1}^{\infty} (1-2^{-n_it})>0$, implying $\frac{\log \prod_{i=1}^k (1-2^{-n_it})}{\sum_{i=1}^kn_i}\to 0$ as $k\to\infty$. Therefore $\hdim (G)=t$. 
From the proof of Proposition \ref{exam}, we choose an appropriate $(n_k)_{k\ge1}$ such that $\pdim (G)=1$.

Using the same notation as Lemma \ref{lem1}, 
let 
 $A(n)=\bigcup_{Q\in\Q\atop Z_n(Q)=1}Q$ and $A=\limsup\limits_{n\to\infty} A(n)$.
We assume that $\gamma_0>0$, since we can replace  $1-\gamma_0$ by an increasing sequence of positive real numbers with limit $1-\gamma_0$ if $\gamma_0=0$. For every $Q\in\Q_n$, the probability $\P(Z_n(Q)=1)$, denoted by $P_n$, does not depend on $Q$.
 Then
 \begin{equation}\label{upper}
 P_n=\P(\bigcup_{m=a_{n-1}}^{a_n-1}\xi_m\in Q)\le \sum_{m=a_{n-1}}^{a_n-1}\P(\xi_m\in Q)=2^{-n}(a_n-a_{n-1})\asymp 2^{-n\gamma_0},
 \end{equation}
 and
 \begin{equation}\label{lower}
 \begin{split}
P_n&\ge \sum_{m=a_{n-1}}^{a_n-1}\P(\xi_m\in Q)-\sum_{m=a_{n-1}}^{a_n-1}\sum_{m'=a_{n-1}\atop m'\ne m}^{a_n-1}\P(\xi_m\in Q, \xi_{m'}\in Q)\\
 &=\sum_{m=a_{n-1}}^{a_n-1}\P(\xi_m\in Q)\Big(1-\sum_{m'=a_{n-1}\atop m'\ne m}^{a_n-1}\P(\xi_m\in Q)\Big)\\
 &\asymp 2^{-n\gamma_0}.
 \end{split}
 \end{equation}
 We can conclude from (\ref{upper}) and (\ref{lower}) that
 \[ \lim_{n\to\infty} \frac{\log_2P_n}{n}=-\gamma_0.\]
 For $Q,Q'\in\Q_n$ with the distance ${\rm dist}(Q,Q')\ge 2^{-n}$, since
\begin{equation}\label{cov}
\begin{split}
{\rm Cov}(Z_n(Q),Z_n(Q'))&=\e(Z_n(Q)Z_n(Q'))-\e(Z_n(Q))\e(Z_n(Q'))\\
&\le \P(Z_n(Q)=1,Z_n(Q')=1)-P_n^2\\
&\le \sum_{m=a_{n-1}}^{a_n-1}\P(\xi_m\in Q)\sum_{m'=a_{n-1}\atop m'\ne m}^{a_n-1}\P(\xi_m\in Q)-P_n^2\\
&\le 2(\sum_{m=a_{n-1}}^{a_n-1}\P(\xi_m\in Q))\sum_{m=a_{n-1}}^{a_n-1}\sum_{m'=a_{n-1}\atop m'\ne m}^{a_n-1}\P(\xi_m\in Q, \xi_{m'}\in Q')\\
& \asymp 2^{-n\gamma_0}\e(Z_n(Q))\e(Z_n(Q')),
\end{split}
\end{equation}
for $\epsilon>0$, we have $f(n,\epsilon)\le3$ for $n$ large enough. Then we have proved that $A$ satisfies the Correlation Condition with $\delta=0$.

 If $\gamma_0=1$, then $\hdim (A)=0$, which implies that $\hdim (A\cap G)=0$. We therefore consider the case $\gamma_0 < 1$.
  Since $\hdim(F\cap A)\le \min\{t, 1-\gamma_0\}$, we prove that $\hdim(F\cap A)\ge \min\{t, 1-\gamma_0\}$ a.s.
  Let $(\epsilon_n)_{n\ge1}$ be a sequence with $\epsilon_n>0$ for all $n$ and such that
$\sum_{n=1}^{\infty}\epsilon_n<1$, and let $(t_n)_{n\ge1}$
be an increasing sequence with $\lim\limits_{n\to\infty}t_n=1-\gamma_0$. By Lemma \ref{lem1}, there is $C\ge 1$ such that for all $n\ge C$, we have
 \[ \P\big(\mathbb{T}=A^{t_n}(n)\big)>1-\epsilon_n,\]
 then using the independence of $(\xi_m)_{n\ge1}$, we derive
 \[\P( \mathbb{T}=A^{t_n}(n){\rm~for~all~} n\ge C)=\prod_{n=C}^{\infty}\P(\mathbb{T}=A^{t_n}(n))>1-\sum_{n=C}^{\infty}\epsilon_n>0.\]

 Fix $\omega\in\{\mathbb{T}=A^{t_n}(n){\rm~for~all~} n\}$, and without losing generality, we suppose $n_1\ge C$. Let $(m_k)_{k\ge1}$ be an increasing sequence of positive integers satisfying $m_k>\sum_{i=1}^{k}n_i$ for $k\ge1$(to be determined later).
 Let $\mathcal{G}'_1=\mathcal{G}_1$. For each $I\in\mathcal{G}'_1$, by Lemma \ref{lem1},
 we can choose a subfamily $\mathcal{A}_1(I):=\{Q\in\Q_{m_1}\colon Q^{t_{m_1}}\subset I\}$ such that $\#\mathcal{A}_1(I)=\lfloor \frac{l_1}{2^{-m_1t_{m_1}}}\rfloor-2${\color{red} ;}
 here we can choose $m_1>\frac{n_1+1}{t_{m_1}}$ such that $\mathcal{A}_1(I)$ exists. Let $\mathcal{F}_1=\{ \mathcal{A}_1(I)\colon I\in \mathcal{G}'_1\}$, then we have $\bigcup_{Q\in\mathcal{G}_1}Q\subset A(m_1)(\omega)$
 and $\#\mathcal{F}_1=\#\mathcal{G}'_1(\lfloor \frac{l_1}{2^{-m_1t_{m_1}}}\rfloor-2)$.

 For $Q\in \mathcal{F}_1$, choose $\lfloor\frac{2^{-m_1}\lfloor2^{n_2t}\rfloor}{ l_1}\rfloor-2$ elements in $\mathcal{G}_2$ which are contained in $Q$. We denote the chosen elements by $\mathcal{G}'_2$,
 that is
 $\mathcal{G}'_2=\{J\in\mathcal{G}_2\colon J\subset Q,~Q\in \mathcal{F}_1\},$
 and $\#\mathcal{G}'_2=(\lfloor\frac{2^{-m_1}\lfloor2^{n_2t}\rfloor}{ l_1}\rfloor-2)\# \mathcal{F}_1$. For $J\in \mathcal{G}'_2$, choose a subfamily $\mathcal{A}_2(J):=\{Q\in\Q_{m_2}\colon Q^{t_{m_2}}\subset J\}$
 such that $\#\mathcal{A}_2(J)=\lfloor \frac{l_2}{2^{-m_2t_{m_2}}}\rfloor-2$, here we can select $m_2>\frac{n_1+n_2+1}{t_{m_2}}$ to guarantee it.
 Write $\mathcal{F}_2=\{\mathcal{A}_2(J)\colon J\in \mathcal{G}'_2\}$, then $\#\mathcal{F}_2=\#\mathcal{G}'_2(\lfloor \frac{l_2}{2^{-m_2t_{m_2}}}\rfloor-2)$.

Continuing inductively, we have families $\mathcal{G}'_k$ and $\mathcal{F}_k$ with
 \begin{equation*}\label{gk}
 \#\mathcal{F}_k=\Big(\Big\lfloor \frac{l_k}{2^{-m_kt_{m_k}}}\Big\rfloor-2\Big)\#\F'_k=N_1\prod_{i=1}^k\Big(\Big\lfloor \frac{l_i}{2^{-m_it_{m_i}}}\Big\rfloor-2\Big)\prod_{i=1}^{k-1}\Big(\Big\lfloor\frac{2^{-m_i}\lfloor2^{n_{i+1}t}\rfloor}{ l_i}\Big\rfloor-2\Big),\\
 \end{equation*}
  and
  \begin{equation*}\label{fk}
  \#\mathcal{G}'_{k+1}=\Big(\Big\lfloor\frac{2^{-m_k}\lfloor2^{n_{k+1}t}\rfloor}{ l_k}\Big\rfloor-2\Big)\#\mathcal{G}_k=N_1\prod_{i=1}^k\Big(\Big\lfloor \frac{l_i}{2^{-m_it_{m_i}}}\Big\rfloor-2\Big)\Big(\Big\lfloor\frac{2^{-m_i}\lfloor2^{n_{i+1}t}\rfloor}{ l_i}\Big\rfloor-2\Big).\\
  \end{equation*}
  We choose $m_k$ and $n_{k+1}$ large enough such that
  \begin{equation}\label{f1}
 \#\mathcal{F}_k  \asymp \frac{2^{(1-t)(\sum_{i=1}^k n_i)}}{2^{(\sum_{i=1}^k m_i(1-t_{m_i}))-m_k}}\quad {\rm and }\quad
   \lim_{k\to\infty}\frac{\log \#\mathcal{F}_k}{-\log 2^{-m_k}}=1-\gamma_0 ,
   \end{equation}
 and 
  \begin{equation}\label{f2}
    \#\mathcal{G}'_{k+1}\asymp \frac{2^{t(\sum_{i=1}^{k+1}n_i)}}{2^{\sum_{i=1}^{k}m_i(1-t_{m_i})}}\quad {\rm and }\quad
  \lim_{k\to\infty}\frac{\log \#\mathcal{G}'_k}{-\log l_k}=t.
  \end{equation}
  Let $F=\bigcap_{k\ge1}\bigcup_{Q\in \F_k}Q=\bigcap_{k\ge1}\bigcup_{I\in \mathcal{G}'_k}I$.
  Then $F\subset G\cap A(\omega)$.

 By (\ref{f1}), (\ref{f2}) and Lemma 2 in \cite{FWW}, we conclude that $$\hdim (F)=\min\{1-\gamma_0, ~t\}.$$ Hence $\hdim(G\cap A)\ge \min\{t, 1-\gamma_0\}$ with positive probability.
  Since $\hdim(G\cap A)\ge \min\{t, 1-\gamma_0\}$ is a tail event, by the Kolmogorov zero-one law, we have $\hdim(G\cap A)= \min\{t, 1-\gamma_0\}$ a.s.

 \end{proof}

\subsection*{Acknowledgements}
We are very grateful to the referees for their valuable comments. The research of Bing Li and Zhang-Nan Hu is supported by  NSFC 11671151, 11871227 and Guangdong Natural Science Foundation 2018B0303110005.
Wen-Chiao Cheng is partially supported by NSC Grant (109-2115-M-034-001).

 { }

\end{document}